\documentclass[a4paper,11pt]{article}

\usepackage[margin=1in]{geometry}
\usepackage{setspace}
\usepackage{amsmath,amssymb,amsthm,mathrsfs}
\usepackage{xcolor}

\usepackage[colorlinks=true,linkcolor=blue,citecolor=blue,urlcolor=blue]{hyperref}

\newtheorem{theorem}{Theorem}[section]
\newtheorem{proposition}[theorem]{Proposition}
\newtheorem{corollary}[theorem]{Corollary}
\newtheorem{definition}[theorem]{Definition}
\newtheorem{example}[theorem]{Example}
\theoremstyle{definition}
\newtheorem{remark}[theorem]{Remark}

\begin{document}

\title{\texorpdfstring{\textbf{On Compact Hessian Manifolds}}{On Compact Hessian Manifolds}}
\author{Hanwen Liu\thanks{ORCID: 0009-0007-2503-9576}\\
{\small Mathematics Institute, University of Warwick}\thanks{University of Warwick, Coventry, United Kingdom, CV4 7AL}\\
{\small \href{mailto:hanwen.liu@warwick.ac.uk}{\texttt{hanwen.liu@warwick.ac.uk}}}}
\date{}
\maketitle

\begin{abstract}
We study compact Hessian manifolds. We prove vanishing of their Euler characteristics. For compact hyperbolic affine manifolds, we deduce that topologically they are periodic mapping tori, and that they admit no Riemannian metric of negative sectional curvature. We then establish a sharp integral inequality characterizing Cheng--Yau metrics. We also construct counterexamples to an incorrect but widely cited assertion of Shima--Yagi.
\end{abstract}

\begin{center}
\textbf{Keywords:} Hessian manifold, hyperbolic affine structure, Cheng-Yau metric.

\textbf{Mathematics Subject Classification (2020):}
Primary 53A15; Secondary 53B12.
\end{center}

{\small\tableofcontents}
\onehalfspacing
\raggedbottom

\section{Introduction and Background}

A Hessian metric is a Riemannian metric which is locally the Hessian of a function with respect to a flat torsion-free connection. The theory originates in the work of Koszul and relates affine differential geometry to convex analysis and information geometry \cite{Koszul,Amari}. The two connections of a Hessian structure are dual with respect to the metric. This duality also relates global 1-form potentials to Euler vector fields.

The topology of a complete Hessian manifold is constrained by convexity. Shima and Yagi proved that its universal cover develops diffeomorphically onto a convex domain \cite{ShimaYagi}. Consequently, every complete Hessian manifold is aspherical and has torsion-free fundamental group. For a compact Hessian manifold, the developing domain is either the whole affine space or is contained in a proper half-space. The theorems of Kostant--Sullivan and Fried then imply that the Euler characteristic vanishes \cite{KostantSullivan,Fried}.

The stronger fibration results concern compact hyperbolic affine manifolds, whose developing domains contain no affine line. Such domains are proper convex cones by Vey's theorem \cite{Vey}. The complete Cheng--Yau metric on a proper cone has a parallel first Koszul form. We prove this directly from logarithmic homogeneity. By a radial deformation of the developing map, we realize every sufficiently nearby canonical class while preserving the developing cone and projective holonomy. Rational approximation follows. Canonical parallelism also gives a periodic mapping-torus description by Welsh's theorem \cite{Welsh}.

We carry out this argument using the canonical metric. The first Koszul form of an arbitrary compact Hessian manifold need not be parallel, contrary to \cite[Theorem~4.1(2)]{ShimaYagi}. Section~\ref{counterexample_section} gives a counterexample on the circle. A parallel affine volume likewise does not force the given affine connection to be the Levi-Civita connection of the prescribed Hessian metric. The valid conclusion is the existence of another Hessian metric whose Levi-Civita connection is the given affine connection \cite{ChengYau,Fox}.

We also establish a sharp integral inequality for the 1-form potential and the first Koszul form, with equality precisely for Cheng--Yau metrics. The finite product cover then gives an obstruction to negative sectional curvature by Preissmann's theorem.

In low dimensions, we classify complete Hessian surfaces and construct Hessian metrics on the mapping tori of periodic diffeomorphisms of closed oriented surfaces. For closed hyperbolic affine 3-manifolds, the fibration theorem gives precisely such a topological description.

Throughout this article, all differentiable manifolds are smooth, connected and without boundary. The affine and Hessian manifolds under consideration are positive-dimensional. Completeness of a Hessian manifold always refers to its Riemannian metric; geodesic completeness of its affine connection is stated separately.

\section{Elementary Topology of Hessian Manifolds}

\subsection{Conjugate connections}

Let $(M,g)$ be a Riemannian manifold. For a connection $\nabla$ on the tangent bundle $TM\rightarrow M$, its conjugate connection $\nabla^*$ with respect to $g$ is defined to be the unique affine connection on the tangent bundle of $M$ such that the equation
$$
Z g(X, Y)=g\left(\nabla_Z X, Y\right)+g\left(X, \nabla_Z^* Y\right)
$$
holds for all smooth vector fields $X, Y, Z $ on $M$. 

\begin{remark}\label{conjugate_connections}
Using the notation introduced above, it is a facile exercise to show that, for any 1-form $ \theta \in \Omega^1(M) $ and any vector field $X$ on $M$, we have $\nabla_X (\theta^\sharp)=(\nabla^*_X \theta)^\sharp$, where here and henceforth $\sharp$ is the musical isomorphism induced by $g$.
\end{remark}

\subsection{Affine structures and Hessian metrics}

An affine structure on a differentiable manifold $M$ is a flat torsion-free connection $\nabla$ on $TM$, or equivalently, a choice of local coordinate system $(x^1,\ldots,x^n)$ on $M$ whose transition functions are affine. The coordinate system $(x^1,\ldots,x^n)$ is termed an affine coordinate system of $M$, and is related to $\nabla$ exactly via the equation $\nabla(dx^i)=0$ for every $i\in\{1,\ldots,n\}$.

A Riemannian metric $h$ on an affine manifold $(M,\nabla)$ is said to be Hessian if every point $p\in M$ admits a neighborhood $U$ on which $h|_U=\nabla(df)$ for some $f\in C^\infty(U)$, or equivalently, the tensor $\nabla h$ is totally symmetric. The pair $((M,\nabla),h)$ is then termed a Hessian manifold, and the function $f$ will be called a local potential of $((M,\nabla),h)$ on $U$. When such a neighborhood $U$ can be chosen to be the entirety of $M$, we say that $f$ is a global potential of $((M,\nabla),h)$.

For a Hessian manifold $((M,\nabla),h)$, the dual structure $\nabla^*$ is the conjugate of $\nabla$ with respect to $h$. The connection $\nabla^*$ is flat and torsion-free, and the averaged connection $$\bar\nabla=\frac{\nabla+\nabla^*}{2}$$ is the Levi-Civita connection of $h$. The Amari--Chentsov tensor is the symmetric tensor $A:=\nabla h$. These statements follow directly in affine coordinates; see also \cite{Amari}.

\subsection{Global Hessian potentials}

A compact Hessian manifold admits no global potential, because, at a maximum of a smooth function, its covariant Hessian is negative semidefinite and hence cannot equal a Riemannian metric.

\begin{definition}\label{Koszul_type}
A Hessian manifold $((M,\nabla),h)$ is said to be exact if $h=\nabla\eta$ for some $\eta\in\Omega^1(M)$, and such an $\eta$ is termed a 1-form potential of $((M,\nabla),h)$.
\end{definition}

A 1-form potential $\eta$ is closed as the antisymmetrization of $\nabla\eta$ is $d\eta$. A 1-form potential of a compact exact Hessian manifold cannot be exact by the preceding maximum value argument. In particular, a compact exact Hessian manifold has positive first Betti number.

\subsection{Convex development and fundamental groups}

On a simply connected affine manifold, a parallel coframe is exact, and integrating such a coframe defines a local affine diffeomorphism $\Phi:M\to\mathbb R^n$, called a developing map, which is uniquely determined up to an affine transformation. For a general affine manifold $(M,\nabla)$, the developing map $\Phi$ is defined on its universal cover $\tilde M$ and satisfies $\Phi(\gamma x)=\rho(\gamma)\Phi(x)$ for all $x\in\tilde M$ and $\gamma\in\pi_1(M)$, where $\rho$ is the affine holonomy representation of $\nabla$.

\begin{theorem}[Shima--Yagi \cite{ShimaYagi}]\label{convex_development}
The developing map of a simply connected complete Hessian manifold is a diffeomorphism onto an open convex domain.
\end{theorem}

\begin{corollary}\label{fundamental_group}
Every complete Hessian manifold is aspherical and has torsion-free fundamental group. Moreover, every compact Hessian manifold has infinite torsion-free fundamental group.
\end{corollary}
\begin{proof}
By Theorem~\ref{convex_development}, for a complete Hessian manifold $((M,\nabla),h)$, the universal cover of $M$ is an open convex domain $\Omega$, and hence is contractible. Moreover, the group $\pi_1(M)$ acts freely on $\Omega$ by affine transformations. Suppose an element $\gamma$ has finite order $m$. Then, for any $x\in\Omega$, the barycenter $$c:=\frac{1}{m}\sum_{k=0}^{m-1}\gamma^k(x)$$ belongs to $\Omega$ by convexity, and is fixed by $\gamma$. Freeness then forces $\gamma$ to be the identity. In the compact case, finiteness of $\pi_1(M)$ would make $\Omega$ compact, which is impossible.
\end{proof}

\subsection{Euler characteristics}

\begin{theorem}\label{Euler}
The Euler characteristic of every compact Hessian manifold is zero.
\end{theorem}
\begin{proof}
For a compact Hessian manifold $((M,\nabla),h)$, by Theorem~\ref{convex_development}, the universal cover of $M$ develops diffeomorphically onto an open convex domain $\Omega\subseteq\mathbb R^n$. If $\Omega=\mathbb R^n$, then the theorem of Kostant--Sullivan applies \cite{KostantSullivan}. Otherwise, convex separation places $\Omega$ in a proper half-space, and the theorem of Fried applies \cite{Fried}.
\end{proof}

\section{Line Bundles and Koszul Forms}

\subsection{Flat line bundles}

A flat line bundle is a real line bundle $\mathscr L$ equipped with a flat connection $D$. Choose parallel local frames satisfying $e_\beta=g_{\alpha\beta}e_\alpha$, where $g_{\alpha\beta}$ is a locally constant cocycle. For a bundle metric $\langle\cdot,\cdot\rangle$ on $\mathscr L$, put $$f_\alpha:=\frac12\log\langle e_\alpha,e_\alpha\rangle.$$ Since $f_\beta-f_\alpha=\log|g_{\alpha\beta}|$, the differentials of the form $df_\alpha$ glue to a closed 1-form, called the first Koszul form of this bundle metric. Its cohomology class, denoted by $\kappa_1(D)$, is independent of the bundle metric and is called the first Koszul class of $(\mathscr L,D)$. An oriented flat line bundle admits a nowhere vanishing parallel global section if and only if $\kappa_1(D)=0$. Indeed, vanishing of this class permits rescaling the bundle metric so that its positive unit section is parallel.

\subsection{The canonical class}

The canonical flat line bundle of an affine manifold $(M,\nabla)$ is $K_\nabla:=(\bigwedge^nT^*M,\nabla)$, where $n:=\dim(M)$ and the same symbol $\nabla$ denotes the induced connection on $\bigwedge^nT^*M$. Triviality of this flat bundle is equivalent to the existence of a nowhere vanishing parallel top form.

For a Hessian manifold $((M,\nabla),h)$, we use the standard convention $$\varkappa:=\frac12d\log(\det(h))$$ for its first Koszul form, where here and henceforth $\det(h)$ will always be the determinant of the coefficient matrix of $h$ in an affine coordinate system of $(M,\nabla)$.
Whenever $M$ is assumed to be orientable, we denote by $\operatorname{vol}_h$ the volume form of $(M,h)$, and we have $\nabla\operatorname{vol}_h=\varkappa\otimes\operatorname{vol}_h$.

We call the class $-\kappa_1(\nabla)$ of $K_\nabla$ the canonical class of $(M,\nabla)$. The metric induced by $h$ on $K_\nabla$ has local coefficient $\det(h)^{-1}$, so its first Koszul form is $-\varkappa$. Therefore, the canonical class of $(M,\nabla)$ is $[\varkappa]$. If the affine developing holonomy of a loop $\gamma\in\pi_1(M)$ has linear part $A_\gamma$, then computation yields
$$
\int_\gamma\varkappa=-\log|\det (A_\gamma)|.
$$
This formula also fixes the holonomy convention used below.

\begin{theorem}[Cheng--Yau \cite{ChengYau,Fox}]\label{parallel_volume_flat}
Let $((M,\nabla),h)$ be a compact oriented Hessian manifold admitting a nonzero parallel top form. Then, there exists a smooth function $f\in C^\infty(M)$ such that the Hessian metric $h_0:=h+\nabla(df)$ satisfies $\nabla h_0=0$. In particular, the affine manifold $(M,\nabla)$ is a Bieberbach manifold and hence has a finite torus cover.
\end{theorem}

\begin{corollary}\label{zero_Koszul}
For a compact Hessian manifold $((M,\nabla),h)$, its first Koszul form vanishes identically if and only if $\nabla$ is the Levi-Civita connection of $h$.
\end{corollary}
\begin{proof}
The reverse implication is immediate. For the forward implication, we pass to the orientation double cover. Since $((M,\nabla),h)$ has its first Koszul form $\varkappa=0$, the volume form of $h$ is parallel, and Theorem~\ref{parallel_volume_flat} gives $h_0=h+\nabla(df)$ with $\nabla h_0=0$. The ratio $r$ of the volume densities of $h$ and $h_0$ is constant. At a maximum and a minimum of $f$, comparison of the two Hessians forces $r=1$. Now, notice that the equation $\det(h+\nabla(df))=\det h$ is uniformly elliptic after linearization along the segment joining $h$ and $h_0$. The maximum principle implies that $f$ is constant. Hence we obtain that $h=h_0$ and $\nabla h=0$.
\end{proof}

\begin{corollary}\label{b1_flat}
For a compact Hessian manifold $((M,\nabla),h)$, its canonical class vanishes if and only if $\nabla$ is the Levi-Civita connection of some flat Riemannian metric. In this case, the manifold $M$ has a finite torus cover. In particular, these conclusions hold if $b_1(M)=0$.
\end{corollary}
\begin{proof}
The reverse implication follows by computing the canonical class using a parallel metric. For the forward implication, write $\varkappa=du$ and pass to the orientation double cover. Then, the top form $e^{-u}\operatorname{vol}_h$ is parallel, so Theorem~\ref{parallel_volume_flat} gives a parallel Riemannian metric. Averaging over the deck group descends such a metric to $M$. Bieberbach's theorem proves the finite-cover assertion, and the last assertion is immediate.
\end{proof}

\section{Counterexamples to an Assertion of Shima--Yagi}\label{counterexample_section}

The assertion in \cite[Theorem~4.1(2)]{ShimaYagi} that the first Koszul form is Levi-Civita parallel on every compact oriented Hessian manifold is false. The same example also disproves the implication from existence of a parallel volume form to equality of the affine and Levi-Civita connections in \cite[Theorem~4.2]{ShimaYagi}.

\begin{example}\label{Shima_counterexample}
On the circle $\mathbb T^1:=\mathbb R/\mathbb Z$, take the translation affine connection $\nabla=d$ and the metric $h=a(x)\,dx^2$, where $$a(x):=\frac{1}{2}+\cos(\pi x)^2.$$ The triple $((\mathbb T^1,\nabla),h)$ is a compact Hessian manifold, with potential $$\frac{x^2}{2}-\frac{\cos(2\pi x)}{8\pi^2}$$ on the universal cover $\mathbb R$ of $\mathbb T^1$. For simplicity, we write $$f:=\frac12\partial_x\log|a|=\frac{a'}{2a}.$$ The first Koszul form $\varkappa$ of $((\mathbb T^1,\nabla),h)$ and the Levi-Civita connection $\bar\nabla$ of $h$ satisfy $\varkappa=f(x)dx$ and $\bar\nabla_{\partial_x}\partial_x=f\partial_x$. Computation yields that the value of $(\bar\nabla_{\partial_x}\varkappa)(\partial_x)$ at the point $x=0$ is
$$
f'(0)-f(0)^2=-\frac{2\pi^2}{3}\ne0.
$$
Thus, the 1-form $\varkappa$ is not parallel.
\end{example}

\begin{remark}
The failed step in the proof on p.~282 of \cite{ShimaYagi} is the identification of the second Koszul forms of the dually flat structures.
\end{remark}

\section{Calculus on Hyperbolic Affine Manifolds}

\subsection{The covariant Legendre transform}

\begin{definition}
A vector field $E$ on an affine manifold $(M,\nabla)$ is an Euler vector field if $\nabla_XE=X$ for every vector field $X$ on $M$. An affine manifold admitting an Euler vector field is said to be radiant.
\end{definition}

Radiance is equivalent to the existence of local affine coordinate systems with linear transition functions. In the developing local coordinate system $(x^1,\dots,x^n)$ on the universal covering of a radiant manifold, an Euler field has the form $$E=\sum_{i=1}^n(x^i-a^i)\partial_i$$ where $\partial_i:=\partial/\partial x^i$ and the affine holonomy fixes its center $a$. Two Euler vector fields on an affine manifold always differ by a parallel vector field.

\begin{proposition}\label{Legendre_duality}
Let $((M,\nabla),h)$ be a Hessian manifold. Then, the following properties hold:
\begin{enumerate}
    \item [1.] If $\eta$ is a 1-form potential of $((M,\nabla),h)$, then $\eta^\sharp$ is an Euler vector field for $(M,\nabla^*)$;
    \item [2.] If $E$ is an Euler vector field for $(M,\nabla)$, then $E^\flat$ is a 1-form potential for $((M,\nabla^*),h)$.
\end{enumerate}
If both $E$ and $\eta$ are given as above, then $\eta^*:=-\eta+\mathcal L_E\eta$ satisfies $\eta^*=E^\flat$ and $\nabla^*\eta^*=h$.
\end{proposition}
\begin{proof}
The first two statements follow from Remark~\ref{conjugate_connections}. The last assertion follows from
$$
(\mathcal L_E\eta)(X)=(\nabla_E\eta)(X)+\eta(\nabla_XE)
=h(E,X)+\eta(X)
$$
for every vector field $X$ on $M$.
\end{proof}

\begin{remark}
The mapping sending $\eta$ to $\eta^*:=-\eta+\mathcal L_E\eta$ in Proposition~\ref{Legendre_duality} is a coordinate-free generalization of the classical Legendre transform.
\end{remark}

An open convex domain is called regular if it contains no affine line. A proper convex cone is a regular open convex domain invariant under all positive homotheties about its vertex. An affine manifold is called hyperbolic if its universal development is a diffeomorphism onto a regular convex domain.

\begin{theorem}[Vey \cite{Vey}]\label{Vey}
If a discrete subgroup of the affine transformation group acts freely, properly discontinuously and cocompactly on a regular convex domain, then the domain is a proper convex cone.
\end{theorem}

Every affine automorphism of a proper convex cone fixes its vertex, so such an automorphism is linear after translating the vertex to the origin. This assertion does not apply to cones containing affine lines.

\begin{proposition}\label{Koszul_hyperbolic}\label{radiant_Koszul}
For a compact Hessian manifold $((M,\nabla),h)$, radiance, exactness and hyperbolicity are equivalent. In each case, its Euler vector field and its 1-form potential are unique.
\end{proposition}
\begin{proof}
Suppose first that $((M,\nabla),h)$ is exact. Write $h=\nabla\eta$ and identify the universal cover of $M$ with a convex domain $\Omega\subseteq\mathbb{R}^n$. Compactness of $M$ gives a constant $C>0$ such that $|\eta|_h\leq C$. Assume that $\Omega$ contains a line with direction vector $v\in\mathbb{R}^n$ and passing through the point $x\in\mathbb{R}^n$. We define $f:=\eta(v)$ and $a(t):=f(x+tv)$. Then, the inequalities $a'(t)=h(v,v)>0$ and $a(t)^2\leq C^2a'(t)$ hold for all $t\in\mathbb R$.

If $a(t_0)>0$ for some $t_0\in\mathbb R$, then differentiation of $b:=1/a$ gives $$b'(t)=-\frac{a'(t)}{a(t)^2}\leq-\frac{1}{C^2}<0$$ for every $t\geq t_0$. Hence, we obtain that $$b(t)\leq b(t_0)-\frac{t-t_0}{C^2},$$ which contradicts positivity of $b(t)$ for large $t$. Reversing the parametrization also excludes any negative value for $a(t_0)$. Thus, we have $a\equiv0$, contradicting the established fact that $a'>0$.

Therefore, exactness implies hyperbolicity, and Theorem~\ref{Vey} shows that hyperbolicity implies radiance.

Suppose now that the manifold is radiant, and let $E$ be its Euler field. Proposition~\ref{Legendre_duality} gives $h=\nabla^*(E^\flat)$, so the dual Hessian structure is exact and hence radiant by the preceding implication. Applying Proposition~\ref{Legendre_duality} to its Euler field gives a 1-form potential for $((M,\nabla),h)$.

Now, we note that, on a compact radiant affine manifold $(M,\nabla)$, there is no nonzero parallel 1-form, because of the existence of the Euler vector field $E$. Indeed, if $\theta\in\Omega^1(M)$ is parallel, then the function $f=\theta(E)$ satisfies $df=\theta$. At a maximum of $f$, the parallel form vanishes and hence vanishes everywhere. Two 1-form potentials differ by a parallel 1-form, proving their uniqueness. The duality in Proposition~\ref{Legendre_duality} then proves uniqueness of the Euler field.
\end{proof}

\subsection{The Cheng--Yau metric}

We use the complete form of the Cheng--Yau theorem, with uniqueness as in \cite{Fox}.

\begin{theorem}[Cheng--Yau \cite{ChengYau,Fox}]\label{ChengYau}
Let $\Omega\subseteq\mathbb R^n$ be a proper convex cone equipped with the standard affine connection $\nabla$. For every $\Lambda>0$, there exists a unique smooth strictly convex function $u$ satisfying 
\begin{equation}\label{CY_equation}
\det(\operatorname{Hess}(u))=\exp(\Lambda u)^2
\end{equation} 
such that the Hessian metric $\nabla(du)$ is complete. The function tends to infinity at the boundary of $\Omega$.
\end{theorem}

The completeness condition is part of the uniqueness statement. We call this function the Cheng--Yau potential with parameter $\Lambda$, and call its Hessian $h:=\nabla(du)$ the Cheng--Yau metric on $\Omega$ with parameter $\Lambda$.

For compact quotients, the parallelism assertion below is also stated in \cite[Lemma~5.4(1)]{Fox}.

\begin{proposition}\label{canonical_homogeneity}\label{canonical_parallelism}
Let $u$ be the Cheng--Yau potential with parameter $\Lambda>0$ on a proper convex cone $\Omega\subseteq\mathbb R^n$, let $h:=\operatorname{Hess}(u)$ be its Hessian metric, and let $E$ be the radial Euler field of $\Omega$. For every linear automorphism $L$ of $\Omega$ and every $t>0$, we have that
$$
u(Lx)=u(x)-\frac1\Lambda\log|\det (L)|,
\qquad
u(tx)=u(x)-\frac n\Lambda\log (t).
$$
Consequently, the metric $h$ and the form $du$ are invariant under all linear automorphisms of $\Omega$. Moreover, the first Koszul form satisfies
$$
\varkappa=\Lambda\,du,\qquad E^\flat=-du,
\qquad\bar\nabla\varkappa=0,
\qquad |\varkappa|_h^2=n\Lambda.
$$
\end{proposition}
\begin{proof}
The function $u\circ L+\Lambda^{-1}\log|\det L|$ satisfies the same Monge--Amp\`ere equation and has complete Hessian metric $L^*h$. Uniqueness gives the first identity. Taking $L=tI$ gives the second identity, and differentiation proves the invariance assertions.

The Monge--Amp\`ere equation gives $\varkappa=\Lambda\,du$. Differentiating logarithmic homogeneity gives $du(E)=-n/\Lambda$ and $h(E,-)=-du$. Dilations preserve $h$, so the vector field $E$ is Killing. Denote by $\nabla^h$ the Levi-Civita connection of $h$. Since the dual form $E^\flat$ is closed, the tensor $\nabla^hE$ is both skew-symmetric and symmetric, and hence vanishes. The remaining statements follow.
\end{proof}

\begin{definition}
A Hessian metric $h$ on a compact affine manifold $(M,\nabla)$ is called a Cheng--Yau metric with parameter $\Lambda>0$ if its first Koszul form satisfies $\nabla\varkappa=\Lambda h$. A compact affine manifold equipped with a Cheng--Yau metric is said to be a Cheng--Yau manifold.
\end{definition}

\begin{theorem}\label{canonical_quotient}
Every compact hyperbolic affine manifold admits a unique Cheng--Yau metric with prescribed parameter $\Lambda>0$. Conversely, every Cheng--Yau manifold $((M,\nabla),h)$ is exact, with its unique 1-form potential $\eta\in\Omega^1(M)$ satisfying
\begin{enumerate}
    \item [1.] $\eta^*=-\eta$;
    \item [2.] $E:=-\eta^\sharp$ is the Euler vector field of $(M,\nabla)$;
    \item [3.] $\varkappa:=\Lambda\eta$ is the first Koszul form of $((M,\nabla),h)$.
\end{enumerate}
\end{theorem}
\begin{proof}
For a compact hyperbolic affine manifold $(M,\nabla)$, Theorem~\ref{Vey} identifies the image of its developing map with a proper convex cone $\Omega$. Proposition~\ref{canonical_homogeneity} descends the canonical Cheng--Yau metric and the differential $du$ of the Cheng--Yau potential of $\Omega$ to the quotient $M=\Omega/\pi_1(M)$.

Conversely, the equation $\nabla\varkappa=\Lambda h$ gives a 1-form potential, so Proposition~\ref{Koszul_hyperbolic} applies. Lift this 1-form potential to the exterior differential $du$ of $u\in C^\infty(\Omega)$. The defining equation for the first Koszul form gives $$\det(\operatorname{Hess}(u))=C\exp(\Lambda u)^2$$ for a positive constant $C>0$. Replacing $u$ by $$f:=u+\frac{\log (C)}{2\Lambda}$$ puts that equation in the form of (\ref{CY_equation}). The lifted Hessian metric $\operatorname{Hess}(f)$ is complete, so uniqueness for the solution to (\ref{CY_equation}) proves the assertion. Proposition~\ref{canonical_parallelism} gives $E=-\eta^\sharp$. Since $\eta(E)=-n/\Lambda$ is constant, Cartan's formula gives $\eta^*=-\eta$.
\end{proof}

Thus, on compact manifolds, a hyperbolic affine structure is equivalent to existence of an exact Hessian metric, recovering Koszul's criterion \cite{KoszulConvexity}.

\subsection{A geometric uncertainty principle}

For a Riemannian metric $h$, we denote its volume density by $d\mu_h$ and its total volume by $\operatorname{Vol}_h(M)$. On an oriented manifold, this density gives the same integrals as the volume form $\operatorname{vol}_h$. The inner product and norm on 1-forms induced by $h$ are denoted by $\langle\cdot,\cdot\rangle_h$ and $|\cdot|_h$.

\begin{theorem}\label{Koszul_uncertainty}
Let $((M,\nabla),h)$ be a compact exact Hessian $n$-manifold with 1-form potential $\eta$ and first Koszul form $\varkappa$. Then
$$
\operatorname{Vol}_h(M)^2
\leq\frac1{n^2}
\left(\int_M|\varkappa|_h^2\,d\mu_h\right)
\left(\int_M|\eta|_h^2\,d\mu_h\right).
$$
Equality holds if and only if $\varkappa=\Lambda\eta$ for a constant $\Lambda>0$, or equivalently, if and only if $h$ is a Cheng--Yau metric.
\end{theorem}
\begin{proof}
Set $A:=\nabla h$ and $V:=\eta^\sharp$. For a local $h$-orthonormal frame $e_1,\ldots,e_n$, Jacobi's determinant formula and the relation $\bar\nabla=(\nabla+\nabla^*)/2$ give
$$
2\varkappa(V)=\sum_{i=1}^{n}A(V,e_i,e_i),
\qquad
(\bar\nabla_X\eta)(Y)=h(X,Y)-\frac12A(X,Y,V).
$$
Here, the second identity holds for all vector fields $X,Y$. Taking its $h$-trace and using symmetry of $A$ yields
\begin{equation}\label{Koszul_divergence}
\operatorname{div}_h(V)=n-\langle\varkappa,\eta\rangle_h.
\end{equation}
The divergence theorem therefore gives
$$
\int_M\langle\varkappa,\eta\rangle_h\,d\mu_h
=n\operatorname{Vol}_h(M)>0.
$$
The Cauchy--Schwarz inequality for the $L^2$ inner product on 1-forms proves the stated bound. Its equality condition is linear dependence by a constant scalar. The positive inner product in the displayed identity forces that scalar to be positive, so equality holds precisely when $\varkappa=\Lambda\eta$ with $\Lambda>0$.

This proportionality implies $\nabla\varkappa=\Lambda h$. Conversely, if $h$ is a Cheng--Yau metric with parameter $\Lambda$, then $\varkappa/\Lambda$ is a 1-form potential. Uniqueness of the 1-form potential in Proposition~\ref{radiant_Koszul} gives $\varkappa=\Lambda\eta$.
\end{proof}

The inequality in Theorem~\ref{Koszul_uncertainty} reveals a geometric uncertainty principle. It gives a lower bound for the product of the $L^2$ norms of the 1-form potential and the first Koszul form. Theorem~\ref{canonical_quotient} supplies a metric attaining equality on every compact affine manifold admitting an exact Hessian metric.

\begin{remark}\label{nonzero_exact_canonical_class}
The canonical class of a compact exact Hessian manifold is nonzero.
Indeed, suppose that $\varkappa=dF$ for some $F\in C^\infty(M)$, and let $\eta$ be a 1-form potential. The product rule and (\ref{Koszul_divergence}) give $\operatorname{div}_h(\exp(F)\eta^\sharp)=n\exp(F)$. The divergence theorem would then imply
$$
0=n\int_M\exp(F)\,d\mu_h>0,
$$
which is a contradiction.
\end{remark}

\section{Fibrations and Deformations}

\subsection{Welsh’s theorem}

We identify the circle with $\mathbb R/\mathbb Z$. For a smooth manifold $F$ and diffeomorphism $\varphi:F\to F$, its mapping torus is $$M_\varphi:=\frac{F\times[0,1]}{(p,1)\sim(\varphi(p),0)}.$$

\begin{proposition}[Welsh \cite{Welsh}]\label{parallel_fibration}
A compact Riemannian manifold admitting a nonzero parallel 1-form is diffeomorphic to the mapping torus of a finite-order diffeomorphism of a connected compact manifold.
\end{proposition}
\begin{proof}
Welsh's theorem gives a bundle over the circle with finite structural group \cite[Proposition~2.1 and Theorem~3.1]{Welsh}. Thus, its monodromy is periodic. If the fiber is disconnected, connectedness makes the monodromy transitive on its components, and its first return to one component gives the required presentation.
\end{proof}

\begin{theorem}\label{periodic_fibration}
Every compact oriented hyperbolic affine manifold is diffeomorphic to the mapping torus of a periodic orientation-preserving diffeomorphism of a connected compact oriented manifold. In particular, such a manifold has a finite smooth cover diffeomorphic to a product with the circle.
\end{theorem}
\begin{proof}
Equip the compact oriented hyperbolic affine manifold with its Cheng--Yau metric. Proposition~\ref{canonical_parallelism} gives a nonzero parallel first Koszul form, so Proposition~\ref{parallel_fibration} applies. The orientations of the total space and circle induce an orientation on the fiber preserved by monodromy. If the monodromy has order $m$, then the $m$-fold cyclic cover is the product of the fiber with the circle.
\end{proof}

\subsection{Deformation of canonical classes}

Koszul proved that hyperbolic affine structures on a compact manifold persist under sufficiently small deformations of the flat torsion-free connection \cite{KoszulDeformation}. The following explicit deformation also controls the canonical class.

A cohomology class in $\mathrm H^1(M;\mathbb R)$ is said to be rational if it lies in the image of the extension of scalars $\mathrm H^1(M;\mathbb Q)\to\mathrm H^1(M;\mathbb R)$.

\begin{theorem}\label{rational_deformation}
Let $(M,\nabla)$ be a compact hyperbolic affine manifold with canonical class $\kappa$. There exist a neighborhood $U$ of zero in $\mathrm H^1(M;\mathbb R)$ and a smooth family $\{\nabla_\alpha\}_{\alpha\in U}$ of hyperbolic affine connections on $M$ such that $\nabla_0=\nabla$ and the canonical class of $(M,\nabla_\alpha)$ is $\alpha+\kappa$ for every $\alpha\in U$. In particular, the connection $\nabla$ is a smooth limit of hyperbolic affine structures with rational canonical class.
\end{theorem}
\begin{proof}
Set $n:=\dim(M)$. Fix smooth closed representatives of a basis of $\mathrm H^1(M;\mathbb R)$, and for each $\alpha\in\mathrm H^1(M;\mathbb R)$, represent $-\alpha/n$ by the corresponding linear combination $\omega$ of these forms. Thus, this $\omega$ depends linearly on $\alpha$ and tends smoothly to zero as $\alpha\to0$.

By Theorem~\ref{Vey}, we may identify the universal cover of $(M,\nabla)$ with a proper convex cone $\Omega\subseteq\mathbb R^n$ whose vertex is the origin. Write $\pi:\Omega\to M$ for the covering map, and $D$ for the standard affine structure of $\mathbb R^n$, and $\rho:\pi_1(M)\to\operatorname{GL}_n(\mathbb R)$ for the holonomy representation. Thus, we have $\pi^*\nabla=D$, and the deck transformation associated to $\gamma\in\pi_1(M)$ sends each $x\in\Omega$ to $\rho(\gamma)x$. We denote by $E$ the Euler vector field on $M$ and, using the same notation, its lift $E_x=x$ on $\Omega$.

Fix $p_0\in\Omega$, and let $F$ be the primitive of $\pi^*\omega$ normalized by $F(p_0)=0$. For each loop $\gamma$ based at $\pi(p_0)$, set
$$
\xi(\gamma):=\int_\gamma\omega.
$$
Closedness of $\omega$ makes $\xi:\pi_1(M)\to\mathbb R$ a group homomorphism, and
\begin{equation}\label{identity_for_F}
F(\rho(\gamma)x)=F(x)+\xi(\gamma)
\end{equation}
for every $x\in\Omega$.

Define $\Psi:\Omega\to\Omega$ by $\Psi(x):=\exp(F(x))x$. By compactness, there exists a neighborhood $U$ of zero such that $|\omega(E)|<1/2$ for every $\alpha\in U$. For each $p\in\Omega$, the function $a(t):=t+F(\exp(t)p)$ satisfies
$$
\Psi(\exp(t)p)=\exp(a(t))p,
\qquad a'(t)>\frac12.
$$
Hence $a$ increases from $-\infty$ to $+\infty$, so $\Psi$ maps every ray bijectively onto itself. Moreover,
$$
d\Psi_x(v)=\exp(F(x))(v+dF_x(v)x)
$$
for $x\in\Omega$ and $v\in T_x\Omega$. Since $1+dF_x(x)>1/2$, this differential is invertible. Thus, $\Psi$ is a diffeomorphism of $\Omega$.

Equation~(\ref{identity_for_F}) gives
$$
\Psi(\rho(\gamma)x)=\rho_\alpha(\gamma)\Psi(x),
\qquad
\rho_\alpha(\gamma):=\exp(\xi(\gamma))\rho(\gamma).
$$
Consequently, $\Psi^*D$ is invariant under the original deck action and descends to a flat torsion-free connection on $M$. Denote this connection by $\nabla_\alpha$. Its developing map $\Psi$ has image $\Omega$, so $\nabla_\alpha$ is hyperbolic. The representations $\rho_\alpha$ and $\rho$ induce the same representation into $\operatorname{PGL}_n(\mathbb R)$. The determinant identity
$$
\log|\det(\rho_\alpha(\gamma))|
=\log|\det(\rho(\gamma))|+\int_\gamma n\omega
$$
and the canonical period formula give the new canonical class $\kappa-n[\omega]=\alpha+\kappa$.

Finally, the normalized primitive $F$ depends linearly on $\omega$. Integration along line segments from $p_0$ gives smooth parameter dependence on every compact subset of $\Omega$. The same holds for $\Psi$ and $\Psi^*D$. Finitely many relatively compact lifts of coordinate neighborhoods covering $M$ therefore give smooth dependence of $\nabla_\alpha$ on $\alpha$. When $\alpha=0$, we have $\omega=0$ and hence $F=0$, so $\Psi=\mathrm{Id}$ and $\nabla_0=\nabla$. Density of rational classes proves the final assertion.
\end{proof}

\subsection{A foliation via lineality space}

We first record two elementary facts about convex domains. For an open convex domain $\Omega\subseteq\mathbb R^n$, its recession cone is
$$
\operatorname{rec}(\Omega):=\{v\in\mathbb R^n:\Omega+tv\subseteq\Omega,\forall t\geq0\}.
$$
A nonzero vector in $\operatorname{rec}(\Omega)$ is termed a recession direction. Recall from convex geometry that the recession cone of an open convex domain is always closed, and every unbounded open convex domain has a nonzero recession direction. Consequently, a convex domain with a trivial recession cone is bounded.

The vector space $W:=\operatorname{rec}(\Omega)\cap-\operatorname{rec}(\Omega)$ is the maximal vector subspace of $\mathbb R^n$ satisfying $\Omega+W=\Omega$, and is termed the lineality space of $\Omega$. Suppose that the lineality space of $\Omega$ is $d$-dimensional. After choosing linear coordinates adapted to $W$, we have $\Omega=\Omega_0\times\mathbb R^d$ for a regular convex domain $\Omega_0\subseteq\mathbb R^{n-d}$.

We now consider the foliation determined by the translation-invariant directions of a convex developing domain. This foliation need not have compact leaves, even when the Hessian manifold is compact.

\begin{proposition}\label{lineality}
Let $((M,\nabla),h)$ be a complete Hessian $n$-manifold whose developing map has image $\Omega\subseteq\mathbb R^n$. Suppose that the lineality space of $\Omega$ is $d$-dimensional, and choose linear coordinates in which $\Omega=\Omega_0\times\mathbb R^d$ with $\Omega_0\subseteq\mathbb R^{n-d}$ regular convex. Then, every affine automorphism $L$ of $\Omega$ is of the form
$$
L(x,y)=(L_0(x),Ax+By+c),
$$
where $A\colon\mathbb R^{n-d}\rightarrow \mathbb R^d$ is linear, and $$L_0\in\operatorname{Aut}(\Omega_0),\qquad B\in\operatorname{GL}(d,\mathbb R),\qquad c\in\mathbb R^d.$$ Moreover, the family $\{\{x\}\times\mathbb R^d:x\in\Omega_0\}$ descends to a foliation of $M$ by immersed Hessian submanifolds.
\end{proposition}
\begin{proof}
The linear part of every affine automorphism of $\Omega$ preserves $W:=\operatorname{rec}(\Omega)\cap-\operatorname{rec}(\Omega)$, which gives the displayed form of $L$. Also, the slices parallel to $W$ are affine submanifolds. Restricting the lifted metric of $h$ and the connection $\nabla$ on $\Omega$ to a slice $\{x\}\times\mathbb R^d$ preserves positive definiteness and total symmetry of the covariant derivative of the metric. Thus, the slices of the form $\{x\}\times\mathbb R^d$ are Hessian. Deck transformations permute these slices, and each leaf is the quotient of a slice by its stabilizer. The induced metric and connection therefore descend to each leaf.
\end{proof}

\section{Low-Dimensional Hessian Geometry}

\subsection{Complete Hessian surfaces}

We use the following two results of Fried--Goldman \cite[Theorem~2.1]{FriedGoldman}: A group acting properly discontinuously by affine transformations on $\mathbb R^2$ is virtually solvable; the same conclusion holds on $\mathbb R^3$ when the action is also cocompact.

\begin{theorem}\label{surface_classification}
A complete Hessian surface is diffeomorphic to the plane, the cylinder, the open M\"obius band, the torus, or the Klein bottle.
\end{theorem}
\begin{proof}
Let $((M,\nabla),h)$ be a complete Hessian surface. By Theorem~\ref{convex_development}, the universal cover of $M$ is a convex domain $\Omega\subseteq\mathbb R^2$, on which $\Gamma:=\pi_1(M)$ acts freely and properly by affine transformations. If $M$ is closed, then Theorem~\ref{Euler} and the classification of closed surfaces give the torus or the Klein bottle.

Suppose now that $M$ is noncompact. We first show that $\Gamma$ is virtually solvable. If $\Omega$ is bounded, then its barycenter is fixed by every affine automorphism, so freeness of the action makes $\Gamma$ trivial. If $\Omega=\mathbb R^2$, then the planar theorem of Fried--Goldman applies. In the remaining case, the recession cone is nonzero and proper, so the cone is a ray, a line, a sector or a half-plane. After passage to a subgroup of index at most two, the linear holonomy preserves a line. Choosing this line as the first coordinate axis makes the linear parts upper triangular. The group of such affine transformations is solvable, because its translation subgroup is Abelian and its group of linear parts is solvable. Thus, we conclude that $\Gamma$ is virtually solvable.

The fundamental group of a noncompact surface is free. A virtually solvable free group is trivial or infinite cyclic, as a free subgroup of rank two would meet a finite-index solvable subgroup in a finite-index non-Abelian free group. The classification of noncompact surfaces with finitely generated fundamental group now gives the plane, the infinite cylinder or the open M\"obius band.
\end{proof}

\subsection{Closed Hessian 3-manifolds}

\begin{theorem}\label{three_dimensional_reduction}
Let $((M,\nabla),h)$ be a closed oriented Hessian 3-manifold. If the image of its developing map contains an affine line, then $\pi_1(M)$ is virtually solvable. If the image of its developing map contains no affine line, then $M$ is diffeomorphic to the mapping torus of a periodic orientation-preserving diffeomorphism of a closed surface of genus at least one.
\end{theorem}
\begin{proof}
Let $\Omega\subseteq\mathbb R^3$ be the image of the developing map of $(M,\nabla)$, and set $\Gamma:=\pi_1(M)$. If $\Omega$ contains no affine line, then Theorem~\ref{periodic_fibration} gives the asserted mapping torus. The homotopy exact sequence of the circle fibration identifies every higher homotopy group of its fiber with that of $M$. Corollary~\ref{fundamental_group} therefore implies that the surface fiber is aspherical, and hence has positive genus.

Suppose that $\Omega$ contains an affine line. By Proposition~\ref{lineality}, we may write $\Omega=\Omega_0\times\mathbb R^d$ with $d>0$ and $\Omega_0$ regular. Projection onto $\Omega_0$ defines a group homomorphism from $\Gamma$ to the affine automorphism group of $\Omega_0$.

If $d=1$, then $\Omega_0$ is regular and planar. When $\Omega_0$ is bounded, write $c$ for its barycenter and $S$ for its covariance matrix, namely
$$
c:=\frac{1}{\operatorname{vol}(\Omega_0)}\int_{\Omega_0}xd\mu(x),
\qquad
S:=\frac{1}{\operatorname{vol}(\Omega_0)}\int_{\Omega_0}(x-c)^{\otimes2}d\mu(x),
$$
where $\mu$ is the Lebesgue measure. Every affine automorphism $\phi(x)=Ax+b$ of $\Omega_0$ satisfies $Ac+b=c$ and $ASA^T=S$. Since $S$ is positive definite, a change of coordinates identifies the affine automorphism group of $\Omega_0$ with a subgroup of $\operatorname O(2)$, which is virtually Abelian. When $\Omega_0$ is unbounded, its recession cone is a ray or a sector containing no line. Its affine automorphism group is virtually solvable by the argument in Theorem~\ref{surface_classification}.

The kernel of the action of $\Gamma$ on $\Omega_0$ consists of transformations that send $(x,y)$ to $(x,\phi(x)+ay)$, where $\phi\colon\mathbb R^2\rightarrow\mathbb R$ is affine and $a\in\mathbb R^\times$. The homomorphism assigning $a$ to such a transformation has Abelian kernel and Abelian image, so this kernel of the action on $\Omega_0$ is solvable, and its extension by the virtually solvable image group is also virtually solvable. This proves the assertion when $d=1$.

If $d=2$, then $\Omega_0$ is an interval, whose affine automorphism group is solvable. Denote by $K$ the kernel of the action of $\Gamma$ on $\Omega_0$, and fix $x\in \Omega_0$. The restriction of $K$ to the plane $\{x\}\times\mathbb R^2$ is faithful, because an element acting trivially on this plane would fix a point of $\Omega$, and hence would be the identity. The restricted action is also proper, since the original action on $\Omega$ is proper. By the theorem of Fried--Goldman, we conclude that $K$ is virtually solvable.

As an extension of virtually solvable groups, the group $\Gamma$ is amenable. Since $\Gamma$ is also finitely generated and linear over $\mathbb R$, Tits's alternative implies that $\Gamma$ is virtually solvable \cite{Tits}.

If $d=3$, then the developing domain is $\mathbb R^3$, and the cocompact theorem of Fried--Goldman directly applies.
\end{proof}

\begin{proposition}\label{periodic_surface_examples}
The mapping torus of every periodic orientation-preserving diffeomorphism of a closed surface of positive genus admits a Hessian structure. For genus at least two, the structure can be chosen to be radiant and exact. For genus one, the structure can be chosen to be flat.
\end{proposition}
\begin{proof}
Let $\Sigma$ be a closed oriented surface of positive genus, and let $\varphi:\Sigma\to\Sigma$ be a periodic orientation-preserving diffeomorphism. When the genus is at least two, averaging a Riemannian metric over the finite group generated by $\varphi$ gives a $\varphi$-invariant conformal structure. Uniformization gives a hyperbolic metric of Gaussian curvature $-1$ in this conformal class, which is also $\varphi$-invariant. Write $\Sigma=\mathbb H^2/\Gamma$, where $\Gamma$ is its group of deck transformations, and denote the hyperbolic metric on $\mathbb H^2$ by $g_{\mathbb H^2}$.

Use the hyperboloid model $\mathbb H^2=\{(x,y,z)\in\mathbb R^3:z>0,\ z^2-x^2-y^2=1\}$. In this model, the orientation-preserving isometry group is $\operatorname{SO}^+(2,1)$, the group of determinant-one linear transformations preserving the Lorentz form and the upper sheet of the hyperboloid. Choose a lift $A$ of $\varphi$ to $\mathbb H^2$. The mapping $A$ belongs to $\operatorname{SO}^+(2,1)$ and normalizes $\Gamma$.

Consider the future Lorentz cone $$C:=\{(x,y,z)\in\mathbb R^3:z>0,\ z^2>x^2+y^2\},$$ equipped with the standard affine structure $\nabla$ of $\mathbb R^3$. Set $Q(x,y,z):=z^2-x^2-y^2$ and define $u:=-\log (\sqrt{Q})$ on $C$. Define the diffeomorphism
$$
\Phi:\mathbb H^2\times\mathbb R\longrightarrow C,\qquad \Phi(p,t):=\exp(t)p.
$$ A direct computation gives $\Phi^*(\nabla(du))=dt^2+g_{\mathbb H^2}$. Under this diffeomorphism, each element $B\in\Gamma$ acts by sending $(p,t)$ to $(Bp,t)$, and the transformation $\exp(L)A$, for any $L>0$, acts by sending $(p,t)$ to $(Ap,t+L)$.
The group generated by these transformations acts freely, properly discontinuously and cocompactly. Since the connection, metric and differential form $du$ are invariant under this group, all three descend to the quotient. The function $-t/L$ modulo $\mathbb Z$ gives a surface fibration with Poincaré map $\varphi$, so the quotient is its mapping torus. The metric has scalar curvature $-2$.

For genus one, the mapping class of $\varphi$ is represented by a finite-order matrix $A\in\operatorname{SL}(2,\mathbb Z)$. Let $m$ be the order of $A$, and let $G$ be the matrix of a positive inner product on $\mathbb R^2$. Averaging gives the $A$-invariant inner product with matrix $$\frac{1}{m}\sum_{k=0}^{m-1}(A^T)^k GA^k.$$ The corresponding Euclidean product metric on $\mathbb R^2\times\mathbb R$ descends to the mapping torus $M_A$. Since $A$ and $\varphi$ represent the same mapping class, we obtain that $M_A$ is diffeomorphic to $M_\varphi$. The metric and its Levi-Civita connection therefore give the required flat Hessian structure.
\end{proof}

\begin{remark}
Every closed flat Riemannian 3-manifold is Hessian, but these manifolds need not fiber over the circle. The oriented Hantzsche--Wendt manifold has its first Betti number $b_1=0$ \cite{HantzscheWendt}. Thus, Theorem~\ref{three_dimensional_reduction} is not a classification of all closed Hessian 3-manifolds by periodic surface bundles. It excludes the Heisenberg nilmanifold from the hyperbolic affine subclass: a periodic surface bundle has a finite product cover, whose fundamental group is either Abelian or contains a surface group of genus at least two, whereas every finite-index subgroup of the integer Heisenberg group is non-Abelian and nilpotent.
\end{remark}

\section{Curvature Obstructions for Hessian Metrics}

\begin{proposition}\label{nonnegative_Ricci_flat}
Let $M$ be a compact manifold admitting a Hessian structure. Every Riemannian metric $h$ on $M$ with nonnegative Ricci curvature is flat.
\end{proposition}
\begin{proof}
By Theorem~\ref{convex_development}, the universal cover $\tilde M$ is contractible. The compact case of the Cheeger--Gromoll splitting theorem \cite{CheegerGromoll} gives an isometric decomposition of $(\tilde M,\tilde h)$ as the product of a compact Riemannian manifold $(N,h_N)$ with a Euclidean space, where $\tilde h$ is the lift of $h$ onto $\tilde M$. Since $N$ is a deformation retract of $\tilde M$, we conclude that $N$ is also contractible. If $\dim N>0$, its top-dimensional homology with coefficients in $\mathbb Z/2\mathbb Z$ is nonzero, contradicting contractibility. Thus, we conclude that $N$ is a point and that both $\tilde h$ and $h$ are flat.
\end{proof}

\begin{theorem}\label{negative_curvature_obstruction}
Let $M$ be a compact manifold of dimension $n\geq2$ admitting a Hessian structure. Suppose that at least one of the following conditions holds:
\begin{enumerate}
\item [1.] $M$ admits a hyperbolic affine structure;
\item [2.] $\pi_1(M)$ is virtually solvable;
\item [3.] $n\leq4$.
\end{enumerate}
Then, the manifold $M$ admits no Riemannian metric of everywhere strictly negative sectional curvature. Under the second condition, every Riemannian metric of nonpositive sectional curvature on $M$ is flat.
\end{theorem}
\begin{proof}
We may pass to the orientation double cover. Under the first condition, Theorem~\ref{periodic_fibration} gives a finite cover diffeomorphic to $F\times S^1$, where the differentiable manifold $F$ is connected, compact and positive-dimensional. Asphericity of $M$ implies that $F$ is aspherical. Its fundamental group is nontrivial, since otherwise $F$ would be a compact contractible manifold of positive dimension. It is also torsion-free, being a subgroup of $\pi_1(M)$. An infinite-order element of $\pi_1(F)$ and the circle generator therefore give
$$
\mathbb Z^2\subseteq\pi_1(F)\times\mathbb Z\subseteq\pi_1(M).
$$
Preissmann's theorem excludes such a subgroup for a compact manifold of strictly negative sectional curvature \cite{Preissmann}.

Under the second condition, the theorem of Gromoll--Wolf \cite[Corollary~2]{GromollWolf} states that a compact Riemannian manifold with nonpositive sectional curvature and virtually solvable fundamental group is flat. This proves both assertions in this case.

If $n=2$, Theorem~\ref{surface_classification} makes $M$ a torus or a Klein bottle, and its fundamental group is virtually Abelian. If $n=3$, Theorem~\ref{three_dimensional_reduction} gives either a virtually solvable fundamental group or a hyperbolic affine structure. These are precisely the two cases already treated. Finally, in dimension four, strictly negative sectional curvature would imply $\chi(M)>0$ by the four-dimensional Gauss--Bonnet positivity theorem of Milnor and Chern; see \cite[Theorem~1.1 and Section~4]{BishopGoldberg}. This contradicts Theorem~\ref{Euler}.
\end{proof}

\begin{corollary}\label{low_dimensional_Einstein}
Let $M$ be a closed manifold admitting a Hessian structure, and suppose that $\dim M\leq4$. Then, every Einstein metric on $M$ is flat.
\end{corollary}
\begin{proof}
The one-dimensional case is immediate. In dimensions two and three, a Riemannian Einstein metric has constant sectional curvature. A negative value is excluded by Theorem~\ref{negative_curvature_obstruction}, while a nonnegative value is covered by Proposition~\ref{nonnegative_Ricci_flat}.

In dimension four, pass to the orientation double cover if necessary and let $g$ be the Einstein metric. The Chern--Gauss--Bonnet formula gives
$$
\chi(M)=\frac{1}{32\pi^2}\int_M(|\operatorname{Rm}_g|_g^2-4|\operatorname{Ric}_g|_g^2+\operatorname{Scal}_g^2)d\mu_g,
$$
where $\operatorname{Scal}_g$ is the scalar curvature, $\operatorname{Ric}_g$ is the Ricci tensor, and $\operatorname{Rm}_g$ is the full Riemann curvature tensor of $(M,g)$. The Einstein equation implies $4|\operatorname{Ric}_g|_g^2=\operatorname{Scal}_g^2$. By Theorem~\ref{Euler}, we have that $\chi(M)=0$, so the integral of $|\operatorname{Rm}_g|_g^2$ is zero. Therefore, the metric $g$ is flat.
\end{proof}

\section*{Acknowledgements}

The author wants to express his most sincere gratitude to Weiyi Zhang for very useful discussions and general advice throughout the preparation of this manuscript, and to Maxwell Stolarski for careful proofreading and various valuable suggestions. Special thanks should also go to Saul Schleimer, who inspired the author's work on Hessian 3-manifolds.

\section*{Statements and Declarations}

\noindent\textbf{Funding.} No funding was received to assist with the preparation of this manuscript.

\noindent\textbf{Competing interests.} The author declares no relevant financial or non-financial interests.

\noindent\textbf{Data availability.} Data sharing is not applicable to this article.

\noindent\textbf{Use of generative AI.} During the preparation of this manuscript, the author used OpenAI’s ChatGPT to assist with mathematical verification, literature checks and language editing. All suggested changes were reviewed by the author, who takes full responsibility for the mathematical accuracy and the final content of this article.

\end{document}